\documentclass[dvips,12pt,a4paper]{amsart}

\usepackage{amsmath,amssymb}

\def\M{{\mathcal{M}}}
\def\H{{\mathcal{H}}}

\def\oM{{\overline{\mathcal{M}}}}

\def\CP{{{\mathbb C}{\rm P}}}

\def\Z{{\mathbb Z}}

\def\Q{{\mathbb Q}}
\def\d{{\partial}}
\def\o{\overline}

\newtheorem{proposition}{Proposition}[section]
\newtheorem{theorem}[proposition]{Theorem}

\newtheorem{lemma}[proposition]{Lemma}

\theoremstyle{definition}

\newtheorem{remark}[proposition]{Remark}

\usepackage{graphicx}

\title[A new proof of Faber's intersection number conjecture]
{A new proof of Faber's intersection number conjecture}

\thanks{A.~B is partially supported by the grants RFBR-07-01-00593, NSh-709.2008.1. Both A.~B. and S.~S. are partly supported by the Vidi grant of NWO}

\author{A.~Buryak}

\address{A.~Buryak:\newline
Department of Mathematics,
University of Amsterdam, \newline
P.~O.~Box 94248, 1090 GE Amsterdam, 
The Netherlands\newline 
\indent and\newline
Department of Mathematics, Moscow State University,\newline
Leninskie gory, 19992 GSP-2 Moscow, Russia}
\email{a.y.buryak@uva.nl, buryaksh@mail.ru}

\author{S.~Shadrin}

\address{S.~Shadrin:\newline
Department of Mathematics,
University of Amsterdam, \newline
P.~O.~Box 94248, 1090 GE Amsterdam, 
The Netherlands\newline 
\indent and\newline
Department of Mathematics, Institute of System Research,\newline
Nakhimovsky prospekt 36-1, Moscow 117218, Russia}
\email{s.shadrin@uva.nl, shadrin@mccme.ru}

\begin{document}

\begin{abstract} We give a new proof of Faber's intersection number conjecture concerning the top intersections in the tautological ring of the moduli space of curves $\M_g$. The proof is based on a very straightforward geometric and combinatorial computation with double ramification cycles. 
\end{abstract}

\maketitle

\tableofcontents

\section{Introduction}

\subsection{Notations}
Let $\M_{g,n}$ be the moduli space of complex algebraic curves of genus $g$ with $n$ labelled marked points. We denote by $\oM_{g,n}$ the space of stable curves which is the Deligne-Mumford compactification of $\M_{g,n}$, and by $\M^{rt}_{g,n}\subset \oM_{g,n}$ the partial compactification of $\M_{g,n}$ by stable nodal curves with rational tails (that is, one irreducible component of a stable curve must still have geometric genus $g$).

Thorughout the paper we work with tautological classes on these spaces. The tautological ring $R^*(\oM_{g,n})$ can be defined as the minimal system of subalgebras of $A^*(\oM_{g,n})$ that contains the classes $\psi_1,\dots,\psi_n$ and is closed under pushforwards with natural maps between moduli spaces. The tautological classes on $\M_{g,n}^{rt}$ are defined as restrictions of the tautological classes on $\oM_{g,n}$.

For further definitions and a detailed discussion of the tautological ring and related topics in geometry of the moduli space of curves we refer the reader to~\cite{Vak}, which is a good survey on the subject.

\subsection{Faber's conjecture}
The conjecture of C.~Faber~\cite{Fab} describes the structure of the tautological ring $R^*(\M_g)$, $g\geq 2$ ($\M_g=\M_{g,0}$). Let us mention the key ingredients of this conjecture.

\begin{enumerate}
\item \emph{(Vanishing)} For any $i\geq g-1$, $R^i(\M_g)=0$.

\item \emph{(Socle)} $R^{g-2}(\M_g) \cong \Q$.

\item \emph{(Perfect pairing)} For any $0\leq i\leq g-2$, the cup product 
$$
R^i(\M_g)\times R^{g-2-i}(\M_g)\to R^{g-2}(\M_g)
$$ 
is a perfect pairing. 

\item \emph{(Top intersections)} Let $\pi\colon \M_{g,n}^{rt}\to\M_g$ be the forgetful morphism. Assume $d_1+\cdots+d_n=g+n-2$, $d_i\geq 1$, $i=1,\dots,n$. Then the class   
$$
\pi_*\left(\prod_{i=1}^n\psi_i^{d_i}(2d_i-1)!!\right)\in R^{g-2}(\M_g)
$$ 
does not depend on $d_1,\dots,d_n$.
\end{enumerate}

The vanishing and socle properties are proven in several different ways, see~\cite{Fab,Loo,GouJacVak}. The perfect pairing is still an open question. The top intersections property, also known as Faber's intersection number conjecture, we discuss in the next Section.

\subsection{Top intersections}

Faber~\cite{Fab} observed that the class $\lambda_g\lambda_{g-1}$ is equal to zero on $\oM_{g,n}\setminus \M_{g,n}^{rt}$, $n\geq 0$. Moreover, the 
linear functional $\int\cdot\lambda_g\lambda_{g-1}\colon R^{g-2}(\M_g)\to\Q$  is an isomorphism. Therefore, a reformulation of the Faber's intersection number conjecture states that
$$
\int_{\oM_{g,n}}\prod_{i=1}^n\psi_i^{d_i}\lambda_g\lambda_{g-1}=\frac{(2g-3+n)!(2g-3)!!}{(2g-2)!\prod_{i=1}^n (2d_i-1)!!}
\int_{\oM_{g,1}}\psi_1^{g-1}\lambda_g\lambda_{g-1}
$$
In this form it is already proved in two different ways that we would like to discuss here.

First proof is based on an observation of Getzler and Pandharipande~\cite{GetPan}. The $\lambda_g\lambda_{g-1}$-integrals appear in the Gromov-Witten theory of $\CP^2$, and the degree zero Virasoro constrains imply Faber's intersection number conjecture. The Virasoro constrains for the Gromov-Witten potential of $\CP^2$ were proved later on by Givental, see~\cite{Giv}.

Second proof is due to Liu and Xu~\cite{LiuXu} via very skillful combinatorial computations. Mumford's formula~\cite{Mum} expresses $\lambda$-classes in terms of $\psi$-, $\kappa$-, and boundary classes. Therefore, the whole problem is reduced to a computation of some non-trivial combinations of the integrals of $\psi$-classes. Witten's conjecture~\cite{Wit} (proved by now in several different ways) allows to compute all integrals of $\psi$-classes using string, dilaton, and KdV equations.

There is a third approach to the same problem due to Goulden, Jackson, and Vakil. They apply relative to infinity localization to the moduli space of mappings to $\CP^1$ in order to obtain relations that involve more general so-called Faber-Hurwitz classes and double Hurwitz numbers in genus $0$. This set of relations allows, in principle, to resolve Faber's intersection number conjecture completely, but there are combinatorial difficulties that they have managed to overcome only for a small number of points. 

We give a new proof of Faber's intersection number conjecture. There are at least two reasons to do that. First, two existing proofs mentioned above involve too advanced technique and, second, they do not provide any geometric feeling for the structure of the tautological ring of $\M_g$. Meanwhile, the approach of Goulden, Jackson, and Vakil allows to understand much more from the low-level geometry of $\M_g$, but it is not a complete proof of the conjecture. Our approach is somewhat similar to the main idea of Goulden, Jackson, and Vakil, but all computations have appeared to be much simpler.

\subsection{Double ramification cycles}

A particular type of double ramification cycles that we need in this paper can be described in the following way. Let $a_1,\dots,a_n$, $n\geq 1$ and $b_1,\dots,b_k$, $0\leq k\leq g$, be positive integers. A subvariety $\H(a_1,\dots,a_n,b_1,\dots,b_k)\subset \M_{g,n+k+1}$ consists of curves $(C_g,x_0,x_1,\dots,x_n,y_1,\dots,y_k)$ such that $-(\sum_{i=1}^n a_i+\sum_{i=1}^kb_i)x_0+\sum_{i=1}^n a_ix_i+\sum_{i=1}^kb_iy_i$ is a principle divisor.
Let $\pi\colon\oM_{g,n+k+1}\to\oM_{g,n+1}$ be the map that forgets the points $y_1,\dots,y_k$. We denote by 
$$
DR_g\left(\prod_{i=1}^n m_{a_i}\prod_{i=1}^k \tilde{m}_{b_i}\right)
$$ 
the push-forward $\pi_*[\H(a_1,\dots,a_n,b_1,\dots,b_k)]$ of the class of the closure of $\H(a_1,\dots,a_n,b_1,\dots,b_k)$ in $\oM_{g,1+n+k}$. Sometimes it is more convenient to consider the restriction of the Po\-in\-ca\-r\'e dual of the class $DR_g(\prod_{i=1}^n m_{a_i}\prod_{i=1}^k \tilde{m}_{b_i})$ to $\M^{rt}_{g,1+n}$; abusing notations we denote it by the same symbol. It is proved in~\cite{ShaZvo} (a generalization of the argument in~\cite{Mum}), that $DR_g(\prod_{i=1}^n m_{a_i}\prod_{i=1}^k \tilde{m}_{b_i})$ has codimension $g-k$.

An advantage of the double ramification cycles is that any tautological class can be expressed in terms of them~\cite{Ion} and there is a simple expression for a $\psi$-class restricted to a DR-cycle in terms of DR-cycles of higher codimension. All DR-cycles lie in the tautological ring~\cite{FabPan}.

The main idea of our approach to Faber's intersection number conjecture can be described in the following way. The fundamental class of the moduli space of curves of genus $g$ can be represented by a DR-class with $k=g$. Then any integral of $\psi$-classes over this cycle can be expressed in terms of integrals over DR-classes with $k=0$ via the same argument as in the standard proof of the string equation. A lemma of E.~Ionel~\cite{Ion} allows to find an expression for any monomial of $\psi$-classes (in $\M_{g,1+n}^{rt}$, $n\geq 1$) in terms of DR-cycles with $n=1$ and $k=0$, that is, $DR_g(m_a)$, $a\geq 2$. This classes are in the socle of the tautological ring of $\M_{g,2}^{rt}$, they are proportional to one particular class $DR_g(m_2)$ which is the hyperelliptic locus generating $R^{g}(\M_{g,2}^{rt})$. 

This gives a combinatorial algorithm to compute explicitely any class involved in Faber's conjecture. A relatively simple and straightforward analysis of this algorithm gives a new prove of Faber's intersection number conjecture.

We hope that the technique of DR-cycles presented here can help with the rest of Faber's conjecture, that is, with the prefect pairing, which is still the most misterious part of it. 

\subsection{Organization of the paper}
We split the argument into geometric (section 2) and combinatorial (sections~3 and~4) parts. In fact, the new ideas in this paper are only in combinatorial computation, while all geometric arguments are a sort of standard routine computations using the space of admissible covers or universal Jacobian. This sort of arguments is rather standard, so we decided to de-emphasize geometric part and we provide only sketches of the proofs there. 
Let us also mention here that all statements in section~3 have a strong geometric flavour in the sense that there are some incomplete geometric arguments that could replace straightforward combinatorial proofs there.

\subsection{Acknowledgements}

The authors are very grateful to I.~Goulden, D.~Jackson, M.~Kazarian, B.~Moonen, R.~Vakil, and D.~Zvonkine for the plenty of fruitful discussions. 

\section{Integrals over DR-cycles}

The goal of this section is to give an algorithm to compute an integral 
$\int_{\oM_{g,n+1}}\lambda_g\lambda_{g-1}\psi_0^0\prod_{i=1}^n\psi_i^{d_i}$ for any non-negative integers 
$d_1,\dots,d_n$ such that $\sum_{i=1}^n d_i=g+n-1$. It is not exactly the integrals we need for Faber's conjecture, however,
 there is an argument of Witten in~\cite{Wit} that explains how to use the string equation in order to recover the integrals 
with arbitrary (positive) powers of $\psi$-classes from these particular ones.

There are two different languages. One can either dicuss the integrals of $\lambda_g\lambda_{g-1}\psi_0^0\prod_{i=1}^n\psi_i^{d_i}$ over DR-cycles in $\oM_{g,1+n}$ (which is usually more convenient for particular computations), or we can say the same for the intersections of $\psi_0^0\prod_{i=1}^n\psi_i^{d_i}$ with the resrtictions of the Poincar\'e duals of DR-cycles to $R^*(\M^{rt}_{g,1+n})$ (which is more convenient for geometric arguments).

We introduce a new notation. Let $d_1,\dots,d_n$, $n\geq 1$, be non-negative integers such that $\sum_{i=1}^n d_i = n-1$. Let $a_1,\dots,a_n$ be arbitrary positive integers. Let 
$$
\left\langle
\prod_{i=1}^n \genfrac[]{0pt}{}{ a_i}{d_i}
\right\rangle^{DR}_g:=\int_{DR_g(\prod_{i=1}^n{m_{a_i}})}\lambda_g\lambda_{g-1}\psi_0^0\prod_{i=1}^n\psi_i^{d_i}.
$$

\subsection{Reduction to initial DR-cycles}
The initial DR-cycles are the cycles with no $\tilde{m}$-s in the notations of the previous section. There is  
a simple reduction formula for $\psi$-classes on the initial DR-cycles that we discuss in the next section. The goal of this section is to express any product of $\psi$-classes in $R^{g+n-1}(\M^{rt}_{g,1+n})$ in terms of the products of $\psi$-classes and initial DR-cycles.

There are two first observations that we are going to use. 

\begin{lemma}\label{lemma:b-2} In $R^0(\M^{rt}_{g,1+n})$ we have: 
$$
DR_g\left(\prod_{i=1}^n m_{a_i}\prod_{i=1}^g \tilde{m}_{b_i}\right)=g!\prod_{i=1}^g b_i^2[\oM_{g,1+n}].
$$
\end{lemma}

\begin{lemma}\label{lemma:proj} Let $\pi\colon \M^{rt}_{1+n}\to\M^{rt}_{1+(n-1)}$ be the map that forgets the last marked point. Assume $k\leq g-1$. Then 
$$
\pi_*DR_g\left(\prod_{i=1}^n m_{a_i}\prod_{i=1}^k \tilde{m}_{b_i}\right)=DR_g\left(\prod_{i=1}^{n-1} m_{a_i}\tilde{m}_{a_n}\prod_{i=1}^k \tilde{m}_{b_i}\right).
$$
\end{lemma}

\begin{proof}[Sketch of proofs]
The first lemma is almost obvious, since the corresponding DR-cycle can be defined via an intersection in the universal Jacobian over $\M^{rt}_{g,1+n}$. Then the lemma follows from the fact that for any curve $C$ of genus $g$ with a chosen base point $x_0$ the map $C^g\to Jac(C)$, $(y_1,\dots,y_g)\mapsto \sum_{i=1}^g b_iy_i$, is of degree $g!\prod_{i=1}^g b_i^g$.
The second lemma follows immediately from the definitions.
\end{proof}

This two lemmas allow to express a monomial of $\psi$-classes in terms of intersections with the initial DR-cycles. 

\begin{proposition}\label{prop:binomial} Let $d_1,\dots,d_n$ be positive integers such that $\sum_{i=1}^n d_i = g+n-1$. For any positive integers $a_1,\dots,a_n$ and $b_1,\dots,b_g$, we have the following identity: 
\begin{align}\label{eq:binomial}
& \left(g! \prod_{i=1}^g b_i^2\right)\cdot \int_{\oM_{g,n+1}}\lambda_g\lambda_{g-1}\psi_0^0 \prod_{i=1}^n \psi_i^{d_i} \\
& = \sum_{\begin{smallmatrix}
I_0\sqcup\cdots \sqcup I_n \\
=\{1,\dots,g\}
\end{smallmatrix}}(-1)^{g-|I_0|}
\left\langle
\prod_{i=1}^n \genfrac[]{0pt}{}{ a_i +\sum_{j\in I_i} b_j}{d_i-|I_i|}
\prod_{i\in I_0} \genfrac[]{0pt}{}{b_i }{ 0}\right\rangle^{DR}_g\notag
\end{align}
\end{proposition}

\begin{proof}[Sketch of a proof]
The argument that  derives this proposition from Lemmas~\ref{lemma:b-2} and~\ref{lemma:proj} is a straighforward application of the pull-back formula for $\psi$-classes, c.~f. proof of string equation in~\cite{Wit}. See~\cite{Sha1,Sha2} for the same argument applied in some other cases that involve DR-cycles.
\end{proof}

\subsection{Expression for a $\psi$-class on the initial cycle}\label{sec:recursion}

In general, an initial DR-cycle is the image of a particular space of admissible covers where one has a map to the 
target genus $0$ curve. A lemma of Ionel~\cite{Ion} states that the $\psi$-class lifted from the DR-cycle is proportional 
to a $\psi$-class lifted from the moduli space of target genus $0$ curves. This allows to use the genus $0$ topological 
recursion relation for a $\psi$-class on double ramification cycles. The result of that can be described on the level of intersection 
numbers by the following proposition.

\begin{proposition} \label{prop:recursion}
For any positive $a,a_1,\dots,a_n$ and for any non-negative $d,d_1,\dots,d_n$, we have the following recursion relation:
\begin{align*}
& a \cdot (2g+n)\cdot \left\langle
\genfrac[]{0pt}{}{ a}{d+1} \prod_{i=1}^n \genfrac[]{0pt}{}{ a_i}{d_i}
\right\rangle^{DR}_g = \\
\notag
& 
\sum_{\begin{smallmatrix}
I\sqcup J =\\
\{1,\dots,n\}
\end{smallmatrix}} \left(
\left(a+\sum\nolimits_{j\in J}a_j\right)\cdot |I|\cdot
\left\langle
\genfrac[]{0pt}{}{ a+\sum_{j\in J}a_j}{0} \prod_{i\in I} \genfrac[]{0pt}{}{ a_i}{d_i}
\right\rangle^{DR}_0
\left\langle
\genfrac[]{0pt}{}{ a}{d} \prod_{j\in J} \genfrac[]{0pt}{}{ a_j}{d_j}
\right\rangle^{DR}_g \right.\\
\notag
& -\left(\sum\nolimits_{j\in J}a_j\right)\cdot (2g+|J|-1)\cdot
\left\langle
\genfrac[]{0pt}{}{ a}{d}
\genfrac[]{0pt}{}{\sum_{j\in J}a_j}{0} \prod_{i\in I} \genfrac[]{0pt}{}{ a_i}{d_i}
\right\rangle^{DR}_0
\left\langle
 \prod_{j\in J} \genfrac[]{0pt}{}{ a_j}{d_j}
\right\rangle^{DR}_g \\
\notag
& 
+\left(a+\sum\nolimits_{j\in J}a_j\right)\cdot \left(2g+|I|\right)\cdot
\left\langle
\genfrac[]{0pt}{}{ a+\sum_{j\in J}a_j}{0} \prod_{i\in I} \genfrac[]{0pt}{}{ a_i}{d_i}
\right\rangle^{DR}_g
\left\langle
\genfrac[]{0pt}{}{ a}{d} \prod_{j\in J} \genfrac[]{0pt}{}{ a_j}{d_j}
\right\rangle^{DR}_0 \\
\notag
& \left. -\left(\sum\nolimits_{j\in J}a_j\right)\cdot (|J|-1)\cdot
\left\langle
\genfrac[]{0pt}{}{ a}{d}
\genfrac[]{0pt}{}{ \sum_{j\in J}a_j}{0} \prod_{i\in I} \genfrac[]{0pt}{}{ a_i}{d_i}
\right\rangle^{DR}_g
\left\langle \prod_{j\in J} \genfrac[]{0pt}{}{ a_j}{d_j}
\right\rangle^{DR}_0 
\right)
\end{align*}
\end{proposition}

Here we use the notation:
\begin{gather*}
\left\langle \prod_{i=1}^n \genfrac[]{0pt}{}{ a_i}{d_i}
\right\rangle^{DR}_0 :=  \int_{DR_0(\prod_{i=1}^nm_{a_i})}\psi_0^0\prod_{i=1}^n\psi_i^{d_i} \\
=\int_{\oM_{0,1+n}}\psi_0^0\prod_{i=1}^n\psi_i^{d_i} =\begin{cases}
\frac{(n-2)!}{d_1!\cdots d_n!},&\text{if $d_1+\cdots+d_n=n-2$},\\ 0,&\text{otherwise}.
\end{cases}
\end{gather*}

\begin{proof}[Sketch of a proof]
This proposition is a very closed relative of the similar formulas in~\cite{Sha1,Sha2} and is based on the Ionel's lemma in the way described above. We only take into account the components of the general expression of a $\psi$-class restricted to a DR-cycle that belong to $\M^{rt}_{g,2+n}$, the rest of the prove is identical to~\cite{Sha1,Sha2}.
\end{proof}

There is a nice interpretation of this recursion in terms of generating vector fields for the intersection numbers over DR-cycles.
Let $\beta$ and $t_{a,d}$, $a\geq 1$, $d\geq 0$, be the formal variables. Define
\begin{align*}
V_g & \colon = \sum_{n=1}^\infty \frac{\beta^{2g+n-1}}{n!} \sum_{\begin{smallmatrix}a_1,\dots,a_n \\ d_1+\cdots+d_n=n-1\end{smallmatrix}}
\left\langle
\prod_{i=1}^n \genfrac[]{0pt}{}{ a_i}{d_i}
\right\rangle^{DR}_g
\prod_{i=1}^n t_{a_i,d_i} \cdot 
\frac{\left(\sum\nolimits_{i=1}^n a_i\right) \partial}{\partial t_{\left(\sum\nolimits_{i=1}^n a_i\right),0}} \\
V_0 & \colon = \sum_{n=2}^\infty \frac{\beta^{n-1}}{n!} \sum_{\begin{smallmatrix}a_1,\dots,a_n \\ d_1+\cdots+d_n=n-2\end{smallmatrix}}
\left\langle
\prod_{i=1}^n \genfrac[]{0pt}{}{ a_i}{d_i}
\right\rangle^{DR}_0
\prod_{i=1}^n t_{a_i,d_i} \cdot 
\frac{\left(\sum\nolimits_{i=1}^n a_i\right) \partial}{\partial t_{\left(\sum\nolimits_{i=1}^n a_i\right),0}}
\end{align*}

Then the recursion relation in proposition~\ref{prop:recursion} can be written as
\begin{equation}\label{eq:recursion}
\frac{a\partial^2V_g}{\partial t_{a,d+1}\partial\beta} = \left[\frac{\partial V_g}{\partial t_{a,d}},\frac{\partial V_0}{\partial \beta}\right]
+ \left[\frac{\partial V_0}{\partial t_{a,d}},\frac{\partial V_g}{\partial \beta}\right].
\end{equation}

\subsection{Initial values}

Using proposition~\ref{prop:recursion} one can eliminate all $\psi$-classes. This reduces the problem of computation of an integral over a DR-cycle to the following set of initial values.

\begin{proposition}\label{prop:initial}
 There is a constant $C_g$ that depends only on genus $g$, such that for any $a\geq 1$
\begin{equation*}
\left\langle\genfrac[]{0pt}{}{ a}{0}\right\rangle^{DR}_g = C_g\cdot \left(a^{2g}-1\right).
\end{equation*}
\end{proposition}

\begin{proof}[Sketch of a proof]
The proof of this proposition is based on the fact that $DR_g(m_a)$ is proportional to a generator of $R^g(\M^{rt}_{g,2})$ with the coefficient $a^{2g}-1$. That can be proved by a universal Jacobian argument, see the proof in~\cite[Proof of Theorem~3.5]{GouJacVak}.
\end{proof}


\section{Basic properties of integrals over DR-cycles}

Here we discuss how the integrals over DR-cycles $DR_g(\prod_{i=1}^nm_{a_i})$ depends on the multiplicities $a_1,\dots,a_n$. 

\subsection{A small simplification}\label{sec:simplification}
Propositions~\ref{prop:recursion} and~\ref{prop:initial} imply that the integral 
$
\left\langle\prod_{i=1}^n \genfrac[]{0pt}{}{ a_i}{d_i}
\right\rangle^{DR}_g
$
is a sum of two rational functions in $a_1,\dots,a_n$ of degree $2g$ and $0$ whose denominators divide $\prod_{i=1}^n a_i^{d_i}$. We know from proposition~\ref{prop:binomial} that in the computation of a particular integral over $\oM_{g,1+n}$ all degree $0$ terms should cancel each other, so we can ignore them in the course of computation. An explicit statement about their values is the following:

\begin{lemma} Let $n\geq 1$. For any non-negative $d_1,\dots,d_n$, $d_1+\cdots+d_n=n-1$, we consider
the degree $0$ part of the expression of the integral 
$
\left\langle\prod_{i=1}^n \genfrac[]{0pt}{}{ a_i}{d_i}
\right\rangle^{DR}_g
$
 as a rational function in $a_1,\dots,a_n$. It
is independent of $a_1,\dots,a_n$ and is equal to 
$-C_g\cdot (n-1)!/d_1!\cdots d_n!.$
\end{lemma}

\begin{proof} It is proved by induction on $n$ via a straightforward application of the recursion relation in proposition~\ref{prop:recursion}.
\end{proof}

One more observation is that all integrals that we consider are proportional to $C_g$, some basic constant that is related to the choice of a particular isomorphism $\int \cdot\lambda_{g}\lambda_{g-1}\colon R^{g-2}\to \Q$. For convenience we may assume that $C_g=1$. Therefore, we can assume for simplicity that the initial values for our computational algorithm are given simply by 
$
\left\langle\genfrac[]{0pt}{}{ a}{0}\right\rangle^{DR}_g = a^{2g}.
$
We keep to this simplified assumption till the end of the paper.

\subsection{Polynomiality}
Taking into account the simplification in section~\ref{sec:simplification} we see that the integral 
$
\left\langle\prod_{i=1}^n \genfrac[]{0pt}{}{ a_i}{d_i}
\right\rangle^{DR}_g
$
is a rational function in $a_1,\dots,a_n$ of degree $2g$ whose denominator divides $\prod_{i=1}^n a_i^{d_i}$.
In fact, one can say more than that.

\begin{proposition} \label{prop:polynomial}
Let $n$ be positive integer. For any non-negative integers $d_1,\dots,d_n$, $d_1+\cdots+d_n=n-1$, the integral 
$
\left\langle\prod_{i=1}^n \genfrac[]{0pt}{}{ a_i}{d_i}
\right\rangle^{DR}_g
$ 
is a polynomial in $a_1,\dots,a_n$.
\end{proposition}

\begin{proof} The proposition in general follows from the particular case when $d_1=n-1$ and $d_2=\cdots=d_n=0$. Indeed, applying the recursion relation in proposition~\ref{prop:recursion} to $\psi$-classes at all points but the first one we come to this particular case, and there is no occurence of $a_1$ in the denominator so far. Hence, the whole integral is a polynomial in $a_1$, and, therefore, in all $a_i$, $i=1,\dots,n$.
So, this special case is enough. It is proven below, in lemma~\ref{lemma:polynomiality} based on lemmas~\ref{lemma:zero} and~\ref{lemma:equiv}. \end{proof}

So, we consider the integral 
$$
I_n(a,a_1,\dots,a_n):=\left\langle\genfrac[]{0pt}{}{ a}{n}\prod_{i=1}^n \genfrac[]{0pt}{}{ a_i}{0}
\right\rangle^{DR}_g, \qquad n\geq 0
$$
It is a homogeneous function of degree $g$ that can be expanded as $I_n=\sum_{i=-n}^{2g} a^i\cdot P_{n,2g-i}(a_1,\dots,a_n)$, where $P_{n,k}$ are some symmetric polynomials of degree $k$ in $n$ variables. Explicit computations with the recursion relation  in proposition~\ref{prop:recursion} give the first few formulas for $I_n$:
\begin{align}
I_0 & = a^{2g}; \label{eq:I0} \\
I_1 & = a^{2g} + \sum_{i=0}^{2g-1} a^i\cdot \frac{2g}{2g+1}\binom{2g+1}{2g-i} a_1^{2g-i}. \label{eq:I1}
\end{align}

\begin{lemma}\label{lemma:zero} For any $n\geq 1$, we have:
$$
I_n(a,a_1,\dots,a_{n-1},0)=I_{n-1}(a,a_1,\dots,a_{n-1}).
$$
\end{lemma}

\begin{proof} This lemma is an exercise on the recursion relation in proposition~\ref{prop:recursion}. We prove it by induction. For $n=1$, it follows from the formula for $I_1$ above. For an arbitrary $n$,
\begin{align} \label{eq:In}
& I_n(a,a_1,\dots,a_{n}) = \\ \notag
& \frac{1}{2g+n} \sum_{i=1}^{n} \frac{a+\sum_{j\not= i}a_j}{a} I_{n-1}(a,a_1,\dots,\hat{a_i},\dots,a_n) \\ \notag
& - \frac{1}{2g+n} \sum_{i<j} \frac{a_i+a_j}{a} I_{n-1}(a,a_1,\dots,\hat{a_i},\dots,\hat{a_j},\dots,a_n,a_i+a_j) \\ \notag
& + \frac{2g}{2g+n} \frac{\left(a+\sum\nolimits_{i=1}^{n}a_i\right)^{2g+1}}{a} 
- \frac{2g}{2g+n} \sum_{i=1}^{n} \frac{a_i^{2g+1}}{a}.
\end{align}
We apply this recursion to $I_n(a,a_1,\dots,a_{n-1},0)-I_{n-1}(a,a_1,\dots,a_{n-1})$. The resulting formula turns to be equal to zero due to the induction assumption.
\end{proof}

This lemma means that $I_n(a,a_1,\dots,a_n)$ splits into the terms that can be expressed in $I_{<n}$ and the terms that are divisible by $a_1\cdots a_n$. For convenience, we introduce a new notation. 
We say that two polynomials in $a_1,\dots,a_n$, $f$ and $g$, are equivalent (notation: $f\equiv g$) if $f-g$ doesn't contain monomials divisible by $a_1\cdots a_n$.

\begin{lemma}\label{lemma:equiv} For any $n\geq 1$, $-n\leq i\leq 2g-n$, we have:
$$
P_{n,2g-i}(a_1,\dots,a_n)\equiv \frac{2g}{n+i}\binom{2g}{i}(a_1+\cdots+a_n)^{2g-i}.
$$
In particular, for $i<0$, $P_{n,2g-i}\equiv 0$.
\end{lemma}

\begin{proof} We prove it by induction on $n$. For $n=1$ is follows from the explicit formula. For $n\geq 2$, equation~\eqref{eq:In} implies that
\begin{align*}
& I_n(a,a_1,\dots,a_n)\equiv \frac{2g}{2g+n}\frac{(a+\sum\nolimits_{i=1}^n a_i)^{2g+1}}{a} \\ 
& -\frac{1}{2g+n}\sum_{i<j} \frac{a_i+a_j}{a}\cdot I_{n-1} (a,a_1,\dots,\hat{a_i},\dots,\hat{a_j},\dots,a_n,a_i+a_j).
\end{align*}
Using the induction assumption, we can continue this equivalence as 
\begin{multline*}
I_n(a,a_1,\dots,a_n)\equiv 
\sum_{i=-1}^{2g-n} a^i\left(\sum\nolimits_{i=1}^n a_i\right)^{2g-i}\cdot \\
\left(\frac{2g}{2g+n}\cdot\binom{2g+1}{i+1}-\frac{n-1}{2g+n}\cdot\frac{2g}{n-1+i+1}\cdot\binom{2g}{i+1}
\right).
\end{multline*}
It is obvious that the coefficient of $a^{-1}\left(\sum\nolimits_{i=1}^n a_i\right)^{2g+1}$ is equal to $0$, and all other coefficients are exactly the same as in the statement of the lemma.
\end{proof}

Finally, we are able to conclude with polynomiality.

\begin{lemma}\label{lemma:polynomiality} For any $n\geq 0$, $I_n(a,a_1,\dots,a_n)$ is a polynomial in $a,a_1,\dots,a_n$.
\end{lemma}

\begin{proof} We know apriori that $I_n$ is a polynomial in $a_1,\dots,a_n$ whose coefficients are polynomials in $a$ and $a^{-1}$
From lemma~\ref{lemma:equiv} we know that $I_n$ is equivalent to a polynomial $\tilde I_n$ in $a_1,\dots,a_n$ whose coefficients are polynomials in $a$. Meanwhile, from lemma~\ref{lemma:zero} we know that $\tilde I_n$ can be chosen in such a way that $I_n-\tilde I_n$ is a linear combination of $I_{<n}$, so one can complete the proof by an induction argument.  
\end{proof}

\subsection{Divisibility}

One more fact about the integrals over DR-cycles that we use below in combinatorial computations is the following:

\begin{proposition} \label{prop:divisibility}
For any non-negative integers $d_1,\dots,d_n$, $d_1+\cdots+d_n=n$, the polynomial in $b,a_1,\dots,a_n$ given by the formula
\begin{equation}\label{eq:divisibility}
\left\langle\genfrac[]{0pt}{}{ b}{0}\prod_{i=1}^n \genfrac[]{0pt}{}{ a_i}{d_i}
\right\rangle^{DR}_g - \sum_{\begin{smallmatrix}j=1 \\ d_j\not= 0\end{smallmatrix}}^n \left\langle\genfrac[]{0pt}{}{ a_j+b}{d_j-1}\prod_{\begin{smallmatrix}i=1 \\ i\not= j\end{smallmatrix}}^n \genfrac[]{0pt}{}{ a_i}{d_i}
\right\rangle^{DR}_g
\end{equation}
is divisible by $b^2$. 
\end{proposition}

\begin{remark}
Observe that using lemma~\ref{lemma:proj} and the pull-back formula for $\psi$-classes one can rewrite this expression as 
$$
\int_{DR_g(\prod_{i=1}^n m_{a_i}\cdot \tilde{m}_b)}\lambda_g\lambda_{g-1}\psi_0^0\prod_{i=1}^n\psi_i^{d_i}.
$$
\end{remark}

\begin{proof} Lemma~\ref{lemma:movepsi} below allows us to consider a special case when $d_1=n$ and $d_2=\cdots=d_n=0$. In this case we have to prove that
$$
I_{n+1}(a,b,a_1,\dots,a_n)-I_{n}(a+b,a_1,\dots,a_n) 
$$
is divisible by $b^2$ (we shift $n$ to $n+1$ for convenience and we use notations from the previous section). 
We can do it by induction on $n$. Explicit formulas~\eqref{eq:I0} and~\eqref{eq:I1} applied for $I_1(a,b)-I_0(a+b)$ prove it for $n=0$. 
Lemma~\ref{lemma:zero} allows to consider only the terms that are divisible by $b\cdot a_1\cdots a_n$. Using lemma~\ref{lemma:equiv} we 
see that it is enough to prove that the linear term in $b$ in the expression
\begin{multline*}
\sum_{i=0}^{2g-n-1} \frac{2g}{n+1+i}\binom{2g}{i}\cdot a^i\cdot\left(b+\sum\nolimits_{j=1}^na_j\right)^{2g-i} \\
-\sum_{i=1}^{2g-n} \frac{2g}{n+i}\binom{2g}{i}\cdot (a+b)^i\cdot \left(\sum\nolimits_{j=1}^na_j\right)^{2g-i}
\end{multline*}
is equal to $0$. The last statement follows from a direct computation.
\end{proof}

\begin{lemma}\label{lemma:movepsi} For any $n\geq 0$, $b,a',a'',a_1,\dots,a_n\geq 0$, $d>0$, $d_1,\dots,d_n\geq 0$, $d+d_1+\dots+d_n=n+1$, we have:
\begin{align*}
& - a'\cdot \left(\left\langle\genfrac[]{0pt}{}{ b}{0}\genfrac[]{0pt}{}{ a'}{d+1}\genfrac[]{0pt}{}{ a''}{0}\prod_{i=1}^n \genfrac[]{0pt}{}{ a_i}{d_i}
\right\rangle^{DR}_g - 
\left\langle\genfrac[]{0pt}{}{ a'+b}{d}\genfrac[]{0pt}{}{ a''}{0}\prod_{i=1}^n \genfrac[]{0pt}{}{ a_i}{d_i}
\right\rangle^{DR}_g \right.\\
& \left. - \sum_{\begin{smallmatrix}j=1 \\ d_j\not= 0\end{smallmatrix}}^n \left\langle\genfrac[]{0pt}{}{ a'}{d+1}\genfrac[]{0pt}{}{ a''}{0}\genfrac[]{0pt}{}{ a_j+b}{d_j-1}\prod_{\begin{smallmatrix}i=1 \\ i\not= j\end{smallmatrix}}^n \genfrac[]{0pt}{}{ a_i}{d_i}
\right\rangle^{DR}_g \right) \\
& + a'' \cdot \left(\left\langle\genfrac[]{0pt}{}{ b}{0}\genfrac[]{0pt}{}{ a'}{d}\genfrac[]{0pt}{}{ a''}{1}\prod_{i=1}^n \genfrac[]{0pt}{}{ a_i}{d_i}
\right\rangle^{DR}_g - 
\left\langle\genfrac[]{0pt}{}{ a'+b}{d-1}\genfrac[]{0pt}{}{ a''}{1}\prod_{i=1}^n \genfrac[]{0pt}{}{ a_i}{d_i}
\right\rangle^{DR}_g \right. \\
&
- \left\langle\genfrac[]{0pt}{}{ a'}{d}\genfrac[]{0pt}{}{ a''+b}{0}\prod_{i=1}^n \genfrac[]{0pt}{}{ a_i}{d_i}
\right\rangle^{DR}_g 
 - \left. \sum_{\begin{smallmatrix}j=1 \\ d_j\not= 0\end{smallmatrix}}^n \left\langle\genfrac[]{0pt}{}{ a'}{d+1}\genfrac[]{0pt}{}{ a''}{0}\genfrac[]{0pt}{}{ a_j+b}{d_j-1}\prod_{\begin{smallmatrix}i=1 \\ i\not= j\end{smallmatrix}}^n \genfrac[]{0pt}{}{ a_i}{d_i}
\right\rangle^{DR}_g \right)
\end{align*}
is divisible by $b^2$.
\end{lemma}

\begin{remark} The meaning of this lemma is that in the proof of proposition~\ref{prop:divisibility} for any $n$ it is enough to consider only one particular choice of $d_1,\dots,d_n$, $d_1+\cdots+d_n=n$.
\end{remark}

\begin{remark} Though this lemma looks a bit combersome and not so natural, in fact it has a clear geometric origin. Indeed, a particular consequence of Ionel's lemma in~\cite{Ion} is that the difference of two $psi$-classes weighted by multiplicities at the corresponding points on one side of a DR-cycle should be a nice expression that doesn't involve any multiplicities coming from the count of simple critical values of the corresponding meromorphic functions.
\end{remark}

\begin{proof}[Proof of lemma~\ref{lemma:movepsi}]
We prove this lemma by induction on $n$. The assumption of induction is that proposition~\ref{prop:divisibility} is true for any number of points that is less than $n+2$. We apply the recursion relation in proposition~\ref{prop:recursion} for the $\psi$-class at the points of multiplicity $a'$ in the first summand and $a''$ in the second summand and collect all terms into the similar sums.

It is convenient to rewrite everything in terms of generating functions defined in section~\ref{sec:recursion}. Let 
$$
U_b:=\frac{\partial}{\partial t_{b,0}} -\beta\sum_{a,d\geq 1} t_{a,d}\frac{\partial}{\partial t_{a+b,d-1}}
$$
Then proposition~\ref{prop:divisibility} can be reformulated as $Lie_{U_b}V_g=O(b^2)$. The statement of this lemma can be reformulated as 
$$
\left(-\frac{a'\partial^2}{\partial t_{a',d+1}\partial t_{a'',0}}+\frac{a''\partial^2}{\partial t_{a',d}\partial t_{a'',1}}\right)Lie_{U_b}V_g=O(b^2).
$$ 
A useful observation is that $Lie_{U_b}V_0=\beta\sum_{a>0}t_{a,0}\frac{(a+b)\partial}{\partial t_{a+b,0}}$. The recursion relation~\eqref{eq:recursion} implies that
$$
-\frac{a'\partial^2}{\partial t_{a',d+1}\partial t_{a'',0}}V_g+\frac{a''\partial^2}{\partial t_{a',d}\partial t_{a'',1}}V_g
=-\left[\frac{\partial V_g}{\partial t_{a',d}},\frac{\partial V_0}{\partial t_{a'',0}}\right]
+\left[\frac{\partial V_g}{\partial t_{a'',0}},\frac{\partial V_0}{\partial t_{a',d}}\right].
$$
Observe also that $[U_b,\frac{\d}{\d t_{a,d}}]=\beta\frac{\d}{\d t_{a+b,d-1}}$ and $Lie_{V_g}U_b=Lie_{V_0}U_b=0$. 

We use these observations in order to obtain the following formulas:
\begin{align*}
& \left(-\frac{a'\partial^2}{\partial t_{a',d+1}\partial t_{a'',0}}+\frac{a''\partial^2}{\partial t_{a',d}\partial t_{a'',1}}\right)Lie_{U_b}V_g = \\
& Lie_{U_b}\left(-\frac{a'\partial^2}{\partial t_{a',d+1}\partial t_{a'',0}}+\frac{a''\partial^2}{\partial t_{a',d}\partial t_{a'',1}}\right)V_g \\
& + \beta\left(\frac{a'\d^2}{\partial t_{a'+b,d}\partial t_{a'',0}}
-\frac{a''\partial^2}{\partial t_{a',d}\partial t_{a''+b,0}}
-\frac{a''\partial^2}{\partial t_{a'+b,d-1}\partial t_{a'',1}}\right)V_g;
\end{align*}
\begin{align*}
& Lie_{U_b} \left(-\left[\frac{\partial V_g}{\partial t_{a',d}},\frac{\partial V_0}{\partial t_{a'',0}}\right]
+\left[\frac{\partial V_g}{\partial t_{a'',0}},\frac{\partial V_0}{\partial t_{a',d}}\right]\right)  = \\
& -\left[\frac{\d}{\d t_{a',d}}Lie_{U_b}V_g,\frac{\d}{\d t_{a'',0}}V_0\right] 
+\left[\frac{\d}{\d t_{a'',0}}Lie_{U_b}V_g,\frac{\d}{\d t_{a',d}}V_0\right] \\
& + \beta\left(\frac{(a''+b)\d^2}{\partial t_{a',d}\partial t_{a''+b,0}}
-\frac{(a'+b)\partial^2}{\partial t_{a'+b,d}\partial t_{a'',0}}
+\frac{a''\partial^2}{\partial t_{a'+b,d-1}\partial t_{a'',1}}\right)V_g.
\end{align*}
Therefore,
\begin{align*}
& \left(-\frac{a'\partial^2}{\partial t_{a',d+1}\partial t_{a'',0}}+\frac{a''\partial^2}{\partial t_{a',d}\partial t_{a'',1}}\right)Lie_{U_b}V_g = \\
& -\left[\frac{\d}{\d t_{a',d}}Lie_{U_b}V_g,\frac{\d}{\d t_{a'',0}}V_0\right] 
+\left[\frac{\d}{\d t_{a'',0}}Lie_{U_b}V_g,\frac{\d}{\d t_{a',d}}V_0\right] \\
& + \beta\cdot b\cdot\left(\frac{\d^2}{\partial t_{a',d}\partial t_{a''+b,0}}
-\frac{\partial^2}{\partial t_{a'+b,d}\partial t_{a'',0}}\right)V_g.
\end{align*}
Here the first two summands in the right hand side are divisible by $b^2$ by induction assumption. Indeed, we are interested in terms of homogeneous degree $n$. In both summands these terms are obtained as some product with the components of $Lie_{U_b}V_g$ of degree $\leq n+1$. The last summand is divisible by $b^2$ for the obvious reason.
\end{proof}

\section{Faber's conjecture}

In this section we apply the properties of the integrals over DR-cycles obtained in the previous sections in order to prove Faber's intersection number conjecture.

\begin{theorem}\label{main theorem} For any positive integers $d_1,\dots,d_n$, $d_1+\cdots+d_n=g+n-2$, we have:
\[  
\int_{\oM_{g,n}}\prod_{i=1}^n\psi_i^{d_i}\lambda_g\lambda_{g-1}=\frac{(2g-3+n)!(2g-3)!!}{(2g-2)!\prod_{i=1}^n (2d_i-1)!!}
\int_{\oM_{g,1}}\psi_1^{g-1}\lambda_g\lambda_{g-1}
\]
\end{theorem}

We prove this theorem in four steps. First, we reformulate Faber's conjecture in a way that is better compatible with DR-cycles (that is, we need a special point with no $\psi$-classes). Second step is an explicit expression of the integral in Faber's conjecture
in terms of coefficients of the polynomials $\left\langle\prod_{i=1}^n \genfrac[]{0pt}{}{a_i}{d_i}\right\rangle^{DR}_g$.
Third step is an explicit formula for these coefficients. Finally, we combine these results into a proof of Faber's conjecture.

\subsection{A reformulation of Faber's conjecture}

There is a string equation for the integrals of $\psi$-classes with $\lambda_g\lambda_{g-1}$ over the moduli space of 
curves (see, e.~g.,~\cite{GouJacVak}). In particular for any positive integers $d_1,\dots,d_n$, $d_1+\ldots+d_n=g+n-1$, we have:

\begin{equation}\label{equivalent formulation}
\int\limits_{\oM_{g,1+n}}\lambda_g\lambda_{g-1}\psi_0^0\prod_{i=1}^n\psi_i^{d_i}=\frac{(2g-2+n)!(2g-1)!!}{(2g-1)!\prod\limits_{i=1}^n(2d_i-1)!!}\int\limits_{\oM_{g,2}}\lambda_g\lambda_{g-1}\psi_0^0\psi_1^g.
\end{equation}

In fact, this equation is equivalent to Faber's conjecture. One can prove that via the same argument as Witten used in~\cite{Wit} for the inversion of string equation.

\subsection{A reformulation of proposition~\ref{prop:recursion}}

We introduce a new notation for the coefficients of the polynomial $\left\langle\prod_{i=1}^m \genfrac[]{0pt}{}{a_i}{d_i}\right\rangle^{DR}_g$. Let 
\[
\left\langle\prod_{i=1}^n \genfrac[]{0pt}{}{a_i}{d_i}\right\rangle^{DR}_g:=
\sum_{
\begin{smallmatrix}
p_1,\dots,p_n \geq 0 \\
p_1+\ldots+p_n=2g 
\end{smallmatrix}
}\left\langle \prod_{i=1}^n \genfrac||{0pt}{}{p_i}{d_i} \right\rangle^{coeff}_g \frac{(2g)!}{p_1!\cdots p_n!} \prod_{i=1}^n a_i^{p_i}.
\]
In this terms, we can rewrite equation~\eqref{eq:binomial} as 
\begin{align}
& \int\limits_{\oM_{g,n+1}}\lambda_g\lambda_{g-1}\psi_0^0\prod_{i=1}^n\psi_i^{d_i}=\label{explicit reformulation}\\
& \frac{(2g)!}{g!2^g}\sum\limits_{
\begin{smallmatrix}
i_0,\dots,i_n \geq 0 \\
i_0+i_1+\cdots+i_n=g \\
i_j\le d_j, j=1,\dots,n
\end{smallmatrix}
}\frac{(-1)^{g-i_0}g!}{i_0!\cdots i_n!}
\left\langle\prod_{j=1}^n
\genfrac||{0pt}{}{2i_j}{d_j-i_j}
\prod\limits_{j=1}^{i_0}
\genfrac||{0pt}{}{2}{0}
\right\rangle^{coeff}_g.\notag
\end{align}
Note that in this formula we use only coefficients
$\left\langle\prod_{i=1}^m\genfrac||{0pt}{}{p_i}{c_i}\right\rangle^{coeff}_g$ with $p_i+c_i\ge 1$, $i=1,\dots, m$.

\subsection{Computation of the coefficients}
We express the coefficients $\left\langle\prod_{i=1}^m\genfrac||{0pt}{}{p_i}{c_i}\right\rangle^{coeff}_g$ in terms of the counting of some paths in the integral lattice.

Consider the lattice $\Z^m$. Let $\{e_1,\dots,e_m\}$ be the standard basis of $\Z^m$.
A path in the space $\Z^m$ is a sequence of points $p_j\in\Z^m$, $j=1,\dots,N$ such that $p_j-p_{j+1}=e_k$ for some $k$.
We associate to each subset $I\subset\{1,\dots,m\}$ a special point in the lattice that we denote by 
$\overline{\mathbf{1}}_I:=\sum_{i\in I}e_i$. 

Consider a point $\overline c=(c_1,\dots,c_m)\in \Z^m$, $c_i\geq 0$, $i=1,\dots,m$. Let
$w_I(\overline c)$ be the number of paths $(p_1,\dots,p_N)$ such that $p_1=\overline c$, $p_N=\overline{\mathbf{1}}_I$, and the points $p_i$, $i=1,\dots,N$ are disjoint from $\overline{\mathbf{1}}_J$ for all $J\not= I$.

\begin{proposition}\label{coefficients of DR-polynomials}
Let $p_1,\dots,p_m$ and $c_1,\dots,c_m$, $m\geq 1$, be non-negative integers such that $p_i+c_i\geq 1$, $i=1,\dots,m$. Then we have:
\[
\left\langle\prod_{i=1}^m\genfrac||{0pt}{}{p_i}{c_i}\right\rangle^{coeff}_g=
\sum\limits_{\begin{smallmatrix}
I\subset\{1,\ldots,m\} \\
I\not=\emptyset
\end{smallmatrix}
}\frac{\prod_{i=1}^{|I|}(2g+i-1)}{\prod_{i\in I}(p_i+c_i)} w_I(\overline{c}).
\]
\end{proposition}

This proposition is based on the following three lemmas that we prove in section~\ref{sec:proofs-lemmas}.

\begin{lemma}\label{lemma:coef1} Let $p_m\ge 1$. Then 
\[
\left\langle\prod_{i=1}^{m-1}\genfrac||{0pt}{}{p_i}{1}\cdot \genfrac||{0pt}{}{p_m}{0}\right\rangle^{coeff}_g=
\prod_{i=1}^{m-1}\frac{2g+i-1}{p_i+1}.
\]
\end{lemma}

\begin{lemma}\label{lemma:coef2}
Let $p_i+c_i\ge 1$, $i=1,\dots,m-2$, and $p_{m-1},p_m\ge 1$. Then
\[
\left\langle\prod_{i=1}^{m-2}\genfrac||{0pt}{}{p_i}{c_i}\cdot\genfrac||{0pt}{}{p_{m-1}+1}{0}\genfrac||{0pt}{}{p_m}{0}\right\rangle^{coeff}_g=
\left\langle\prod_{i=1}^{m-2}\genfrac||{0pt}{}{p_i}{c_i}\cdot\genfrac||{0pt}{}{p_{m-1}}{0}\genfrac||{0pt}{}{p_m+1}{0}\right\rangle^{coeff}_g.
\]
\end{lemma}

\begin{lemma}\label{lemma:coef3}
Let $p_i+c_i\ge 1$, $i=1,\dots,m-1$. Then
\[
\left\langle\prod_{i=1}^{m-1}\genfrac||{0pt}{}{p_i}{c_i}\cdot\genfrac||{0pt}{}{1}{0}\right\rangle^{coeff}_g=
\sum_{i=1}^{m-1}
\left\langle\prod_{\begin{smallmatrix}j=1 \\ 
j\not= i
\end{smallmatrix}}^{m-1}\genfrac||{0pt}{}{p_j}{c_j}\cdot\genfrac||{0pt}{}{p_i+1}{c_i-1}\right\rangle^{coeff}_g
\]
\end{lemma}

\begin{proof}[Proof of proposition~\ref{coefficients of DR-polynomials}]
Since $\left\langle\prod_{i=1}^m\genfrac||{0pt}{}{p_i}{c_i}\right\rangle^{coeff}_g\not=0$ only for $\sum_{i=1}^m c_i=m-1$, we have at least one of the indices $c_i$ equal to zero. Assume that there exactly one index equal to zero, say, $c_i=0$. Then all other indices $c_j$, $j\not=i$, are equal to $1$. In this case, the proposition follows from lemma~\ref{lemma:coef1}. Indeed, in this case $w_I(\overline{c})$ is equal to $0$ for all $I\subset\{1,\dots,m\}$ except for $I=\{1,\dots,m\}\setminus\{i\}$, where $w_I(\overline{c})=1$.

If we have at least two zeros among the indices $c_i$, $i=1,\dots,m$, we can apply the following corollary of lemmas~\ref{lemma:coef2} and~\ref{lemma:coef3}. If $p_i+c_i\ge 1$ for $i=1,\dots,m-2$, and
$p_{m-1},p_m\ge 1$, then
\begin{align}\label{identity:coefreduction}
& \left\langle\prod_{i=1}^{m-2}\genfrac||{0pt}{}{p_i}{c_i}\cdot\genfrac||{0pt}{}{p_{m-1}}{0}\genfrac||{0pt}{}{p_m}{0}\right\rangle^{coeff}_g= \\
& \notag \sum_{i=1}^{m-2}
\left\langle\prod_{\begin{smallmatrix}j=1 \\ 
j\not= i
\end{smallmatrix}}^{m-2}\genfrac||{0pt}{}{p_j}{c_j}\cdot
\genfrac||{0pt}{}{p_i+1}{c_i-1}
\genfrac||{0pt}{}{p_{m-1}+p_{m}-1}{0}\right\rangle^{coeff}_g
\end{align} 
This relation is compatible with the definition of the number of paths. Applying this relation sufficiently many times we come to the situation when all indices but one are equal to $1$. This corresponds to a point $\overline{\mathbf{1}}_I$ for some $I$ in the lattice $\mathbb{Z}^m$, and lemma~\ref{lemma:coef1} implies that the coefficient at this endpoint is exactly $\frac{\prod_{i=1}^{|I|}(2g+i-1)}{\prod_{i\in I}(p_i+c_i)}$. 
\end{proof}

\subsection{A proof of Faber's conjecture} In this section, we prove Faber's intersection number conjecture.

\begin{proof}[Proof of theorem~\ref{main theorem}] We are going to compute explicitely both side of equation~(\ref{equivalent formulation}) using proposition~\ref{coefficients of DR-polynomials}.

We denote by $\overline{i}$ the vector $(i_1,\dots,i_n)\in\Z^n$. Proposition~\ref{coefficients of DR-polynomials} and equation~(\ref{explicit reformulation}) imply that
\begin{align}\label{main sum}
& \int\limits_{\oM_{g,n+1}}\lambda_g\lambda_{g-1}\psi_0^0\prod_{i=1}^n\psi_i^{d_i}=\\
& \frac{(2g)!}{g!2^g}\sum\limits_{
\begin{smallmatrix}
i_0,\dots,i_n \geq 0 \\
i_0+i_1+\cdots+i_n=g \\
i_j\le d_j,\ j=1,\dots,n \\
I\subset \{1,\dots,n\},\ I\not=\emptyset
\end{smallmatrix}
}\frac{(-1)^{g-i_0}g!}{i_0!\cdots i_n!}
\frac{\prod_{j=1}^{|I|}(2g+j-1)}{\prod_{j\in I}(d_j+i_j)} w_I(\overline{d}-\overline{i}).\notag
\end{align}

This allows us to compute the integral in the right hand side of equation~\eqref{equivalent formulation}. Indeed,
\begin{equation}
\int\limits_{\oM_{g,2}}\lambda_g\lambda_{g-1}\psi_1^g \label{one point}
= \frac{(2g)!}{g!2^g}\sum\limits_{i=0}^{g}(-1)^i\binom{g}{i}\frac{2g}{g+i} 
 = \frac{g!}{2^{g-1}}.
\end{equation}

Equations~\eqref{main sum} and~\eqref{one point} imply that equation~\eqref{equivalent formulation} is
equivalent to
\begin{align}\label{big sum}
& \sum\limits_{
\begin{smallmatrix}
i_0,\dots,i_n \geq 0 \\
i_0+i_1+\cdots+i_n=g \\
i_j\le d_j,\ j=1,\dots,n \\
I\subset \{1,\dots,n\},\ I\not=\emptyset
\end{smallmatrix}
}\frac{(-1)^{g-i_0}}{i_0!\cdots i_n!}
\frac{\prod_{j=1}^{|I|}(2g+j-1)}{\prod_{j\in I}(d_j+i_j)} w_I(\overline{d}-\overline{i}) \\
& =\prod\limits_{i=1}^{n-1}(2g+i-1)\prod\limits_{i=1}^n\frac{(d_i-1)!}{(2d_i-1)!}.\notag
\end{align}

We prove in lemma~\ref{lemma:vanishing} below that for all subsets $I\subset\{1,\dots,n\}$ such that $|I|\leq n-2$
the corresponding summands on the left hand side of this formula vanish. Before that, let us introduce a new definition that would allow us to count the number of paths in the lattice in a convenient way. 

Let $\o c\in\Z^n$. We denote by $w_0(\o c)$ the number of paths $(p_1,\dots,p_N)$ in $\Z^n$, such that $p_1=\o c$ and $p_N=(0,\dots,0)$. Observe that for any non-empty $I\subset \{1,\dots,n\}$,
\begin{align*}
w_I(\overline c) & =\begin{cases}
                   1,&\text{if $\o c=\overline{\mathbf{1}}_I$},\\
                   \sum_{k\in I} w_0(\o c-\overline{\mathbf{1}}_I-\overline{\mathbf{1}}_{\{k\}}),&\text{otherwise}.
\end{cases} \\
w_0(\o c) & =\begin{cases}
             (\sum_{i=1}^nc_i)!/\prod_{j=1}^nc_j!,&\text{if $c_i\ge 0$}\\
             0,&\text{otherwise}.
           \end{cases}
\end{align*}
We also introduce two auxiliary functions. Let 
\begin{align*}
f_{a,b}(x) & :=\sum_{i=0}^a(-1)^i\frac{x^{a-i}}{b+i}\binom{a}{i}, \\ 
g_{a,b}(x) & :=\int x^a (1-x)^b dx. 
\end{align*}
We list some properties of these functions:
\begin{align*}
f_{a,b} & =(-1)^{a+1}x^{a+b}g_{-a-b-1,a}, & \frac{d}{dx}f_{a,b} & =af_{a-1,b},\\
g_{a,b}(1) & =\frac{b!}{(a+1)(a+2)\cdots(a+b+1)}, & f_{a,b}(0) & =\frac{(-1)^a}{a+b}.
\end{align*}

\begin{lemma}\label{lemma:vanishing}
Let $I$ be a subset of $\{1,\dots,n\}$ such that $0<|I|\le n-2$. Then the corresponding summand of the left hand side of equation~\eqref{big sum} is equal to zero, that is,
\[
\sum\limits_{
\begin{smallmatrix}
i_0,\dots,i_n \geq 0 \\
i_0+i_1+\cdots+i_n=g \\
i_j\le d_j,\ j=1,\dots,n 
\end{smallmatrix}
}\frac{(-1)^{g-i_0}}{i_0!\cdots i_n!}
\frac{w_I(\overline{d}-\overline{i})}{\prod_{j\in I}(d_j+i_j)} =0
\]
\end{lemma}

\begin{proof}
Let $k\in I$. An explicit calculations shows that  
\begin{align*}
&\sum\limits_{
\begin{smallmatrix}
i_0,\dots,i_n \geq 0 \\
i_0+i_1+\cdots+i_n=g \\
i_j\le d_j,\ j=1,\dots,n 
\end{smallmatrix}
}\frac{(-1)^{g-i_0}}{i_0!\cdots i_n!}
\frac{w_0(\overline{d}-\overline{i}-\overline{\mathbf{1}}_I-\overline{\mathbf{1}}_{\{k\}})}{\prod_{j\in I}(d_j+i_j)} = \\
&\left.\left(\frac{d}{dx}\right)^{n-2-|I|}\left(
\frac{f_{d_k-2,d_k}}{(d_k-2)!}
\prod_{{\begin{smallmatrix}j\in I\\ j\not= k \end{smallmatrix}}}
\frac{f_{d_j-1,d_j}}{(d_j-1)!}
\prod_{j\notin I}
\frac{(x-1)^{d_j}}{d_j!}\right)\right|_{x=1}
\end{align*}
The derivative at $x=1$ in the right hand side of this equation is equal to zero. Indeed, $(x-1)$ enters the numerator with the multiplicity
$\sum_{j\not\in I}d_j\ge\sum_{j\not\in I}1=n-|I|>n-2-|I|$.

In order to complete the proof, we just observe that the sum over all $k\in I$ of the left hand side of this formula is exactly the expression in the statment of the lemma.
\end{proof}

This lemma implies that the left hand side of equation~\eqref{big sum} is equal to $S_0+\sum_{l=1}^nS_l$, where
\begin{align*}
S_0& =
\sum\limits_{
\begin{smallmatrix}
i_0,\dots,i_n \geq 0 \\
i_0+i_1+\cdots+i_n=g \\
i_j\le d_j,\ j=1,\dots,n 
\end{smallmatrix}
}\frac{(-1)^{g-i_0}}{i_0!\cdots i_n!}
\frac{\prod_{j=1}^{n}(2g+j-1)}{\prod_{j=1}^n(d_j+i_j)} w_{\{1,\dots,n\}}(\overline{d}-\overline{i})
,\\
S_l & =
\sum\limits_{
\begin{smallmatrix}
i_0,\dots,i_n \geq 0 \\
i_0+i_1+\cdots+i_n=g \\
i_j\le d_j,\ j=1,\dots,n 
\end{smallmatrix}
}\frac{(-1)^{g-i_0}}{i_0!\cdots i_n!}
\frac{\prod_{j=1}^{n-1}(2g+j-1)}{\prod_{j\not=l}(d_j+i_j)} w_{\{1,\dots,n\}\setminus\{l\}}(\overline{d}-\overline{i})
\end{align*}

Recall that $\sum_{i=1}^n d_i=g+n-1$. Using the expression of $w_I$ in terms of $w_0$ in this particular case, we see
that $S_0$ can be represented in the following way: 
\begin{align*}
S_0 & =\prod\limits_{j=1}^n\frac{2g+j-1}{(d_j-1)!}\int_0^1\prod\limits_{j=1}^n f_{d_j-1,d_j}dx \\
& =
\frac{\prod_{j=1}^{n-1}(2g+j-1)}
{\prod_{j=1}^n(d_j-1)!}(-1)^{g+n-1}\int_0^1\prod_{j=1}^n g_{-2d_j,d_j-1}dx^{2g+n-1} \\
& =\prod_{i=1}^{n-1}(2g+i-1)\prod_{i=1}^n\frac{(d_i-1)!}{(2d_i-1)!}\\
&-\frac{\prod_{j=1}^{n-1}(2g+j-1)}{\prod_{j=1}^n(d_j-1)!}
\sum_{l=1}^n
(-1)^{d_l}
\int_0^1(1-x)^{d_l-1}\prod\limits_{j\ne l}f_{d_j-1,d_j}dx.
\end{align*}

In order to complete the proof of equation~\eqref{big sum}, and, therefore, the proof of theorem~\ref{main theorem}, it is sufficient to show that
\[
S_l=\frac{\prod_{j=1}^{n-1}(2g+j-1)}{\prod_{j=1}^n(d_j-1)!}(-1)^{d_l}\int_0^1(1-x)^{d_l-1}\prod_{j\ne l}f_{d_j-1,d_j}dx
\]
for all $l=1,\dots,n$. The right hand side of this formula is equal to
\begin{align*}
& -\frac{\prod_{j=1}^{n-1}(2g+j-1)}{d_l!\prod_{j\ne l}(d_j-1)!}
\int_0^1\prod_{j\ne l}f_{d_j-1,d_j}d(x-1)^{d_l} = \\
& \sum_{k\ne l}\frac{\prod_{j=1}^{n-1}(2g+j-1)}{d_l!(d_k-2)!
\prod_{j\ne l,k}(d_j-1)!}
\int_0^1(x-1)^{d_l}f_{d_k-2,d_k}
\prod_{j\ne l,k}f_{d_j-1,d_j}dx \\
& +\frac{(-1)^g}{d_l!\prod_{j\ne l}(d_l-1)!}\frac{\prod_{j=1}^{n-1}(2g+j-1)}{\prod_{j\ne l}(2d_j-1)}.
\end{align*}
Meanwhile, using the expression of $w_I$ in terms of $w_0$, we see that 
\begin{align*}
S_l & = 
 \sum\limits_{
\begin{smallmatrix}
i_0,\dots,i_n \geq 0 \\
i_0+i_1+\cdots+i_n=g \\
i_j\le d_j,\ j=1,\dots,n 
\end{smallmatrix}
}
\left(
\frac{(-1)^{g-i_0}\prod_{j=1}^{n-1}(2g+j-1)}{\prod_{j=0}^n i_j!\prod_{j\ne l}(d_j+i_j)} \right. \\
& \phantom{
 \sum\limits_{
\begin{smallmatrix}
i_0,\dots,i_n \geq 0 \\
i_0+i_1+\cdots+i_n=g \\
i_j\le d_j,\ j=1,\dots,n 
\end{smallmatrix}
} {\ }}
\left. \cdot \sum_{k\ne l} 
w_0(\overline{d}-\overline{i}-\overline{\mathbf{1}}_{\{1,\dots,n\}\setminus\{l\}}-\overline{\mathbf{1}}_{\{k\}}) \right)  \\
& + \frac{(-1)^g\prod_{j=1}^{n-1}(2g+j-1)}{d_l!\prod_{j\ne l}(d_l-1)!\prod_{j\ne l}(2d_j-1)},
\end{align*}
and an explicit calculations shows that
\begin{align*}
& \sum\limits_{
\begin{smallmatrix}
i_0,\dots,i_n \geq 0 \\
i_0+i_1+\cdots+i_n=g \\
i_j\le d_j,\ j=1,\dots,n 
\end{smallmatrix}
}
\frac{(-1)^{g-i_0}
w_0(\overline{d}-\overline{i}-\overline{\mathbf{1}}_{\{1,\dots,n\}\setminus\{l\}}-\overline{\mathbf{1}}_{\{k\}})}
{\prod_{j=0}^ni_j!\prod_{j\ne l}(d_j+i_j)} \\
& =\frac{\int_0^1(x-1)^{d_l}f_{d_k-2,d_k}\prod_{j\ne l,k}f_{d_j-1,d_j}dx}{d_l!(d_k-2)!\prod_{j\ne l,k}(d_j-1)!}.
\end{align*}
\end{proof}

\subsection{Proofs of lemmas~\ref{lemma:coef1}--~\ref{lemma:coef3}}\label{sec:proofs-lemmas}

\begin{proof}[Proof of lemma~\ref{lemma:coef1}]
Let us reformulate the lemma. We want to prove that the coefficient of the monomial $a_1^{p_1}\ldots a_m^{p_m}$
in $\left\langle\prod_{i=1}^{m-1}\genfrac[]{0pt}{}{a_i}{1}\genfrac[]{0pt}{}{a_m}{0}\right\rangle^{DR}_g$
is equal to the coefficient of the same monomial in 
\begin{equation}
\frac{2g}{2g+m-1}\frac{(a_1+\ldots+a_m)^{2g+m-1}}{a_1\ldots a_{m-1}}.\label{ratfunc}
\end{equation}

We prove it by induction on $m$. The base of induction, $m=1$, is obvious.  Let $m\ge 2$. The induction assumption and the recursion relation 
in proposition~\ref{prop:recursion}
implies that the coefficient of the monomial $a_1^{p_1}\cdots a_m^{p_m}$
in $\left\langle\prod_{i=1}^{m-1}\genfrac[]{0pt}{}{a_i}{1}\genfrac[]{0pt}{}{a_m}{0}\right\rangle^{DR}_g$
is equal to the coefficient of the same monomial in
\begin{align*}
& \sum_{I\subset\{2,\ldots,m-1\}}\left (\frac{2g+m-|I|-2}{2g+m-1}\cdot \frac{a_1+a_m+\sum_{i\in I}a_i}{a_1}\cdot \right. \\
& \phantom{ \sum_{I\subset\{2,\ldots,m-1\}} } \left. \frac{2g}{2g+m-|I|-2}\cdot 
\frac{\left(\sum_{i=1}^m a_i\right)^{2g+m-|I|-2}}{\prod_{i\in\{2,\ldots,m-1\}\backslash I}a_i}\cdot |I|!\right)=\\
& =\frac{2g}{2g+m-1}\cdot \frac{1}{\prod_{i=1}^{m-1}a_i}\cdot \\
& \phantom{ = }
\sum_{I\subset\{2,\ldots,m-1\}}|I|!\left(a_1+a_m+\sum\limits_{i\in I}a_i\right)\left(\prod_{i\in I}a_i\right)\left(\sum_{i=1}^m a_i\right)^{2g+m-|I|-2}.
\end{align*}
The right hand side of this equation coincides with~\eqref{ratfunc} by the following combinatorial observation:
\[
\sum_{I\subset\{2,\ldots,k\}}|I|!\left(x_1+\sum_{i\in I}x_i\right)\left(\prod_{i\in I}x_i\right)\left(\sum_{i=1}^k x_i\right)^{k-|I|-1}=\left(\sum_{i=1}^k x_i\right)^k
\]
(we assume $k\geq 1$).
\end{proof}

\begin{proof}[Proof of Lemma \ref{lemma:coef2}]
First, let us introduce some notations. Let $P$ and $Q$ be polynomials in the variables $a_1,\dots,a_m$. Let $I\subset\{1,\dots,m\}$. We write $P\stackrel{a_I}{\equiv}Q$ iff the polynomial $P-Q$ doesn't have monomials divisible by $\prod_{i\in I}a_i$.
Let $J\subset\{1,\dots,m\}$. We will write $P(a_J)$ in order to specify that the polynomial $P$ depends only on variables $a_i$ for $i\in J$.
 
The lemma is equivalent to the following statement. If $c_1,\dots,c_m, c_1+\ldots+c_m=m-1$ are non-negative integers and
$E=\{i\in\{1,\dots,m\}|c_i=0\}$, then there exists a polynomial $P(a_{\{1,\dots,m\}\setminus E},x)$ such that
\begin{equation}\label{formula:empty points}
\left\langle\prod_{i=1}^m\genfrac[]{0pt}{}{a_i}{c_i}\right\rangle^{DR}_g\stackrel{a_E}{\equiv}P\left(a_{\{1,\dots,m\}\setminus E},\sum_{i\in E}a_i\right).
\end{equation} 
We prove this by induction on $m$. The base of induction $m=1$ is obvious. Let $m\ge 2$. If $|E|=1$ then there is nothing to prove. Suppose that $|E|\ge 2$. Then there exists $i\in I_m$ such that 
$c_i\ge 2$. Without loss of generality we can assume that $i=m$. We are going to apply the recursion relation in proposition~\ref{prop:recursion}.

Let $m=n+1$, $c_m=d+1$, $a_m=a$, $d_j=c_j$, $1\le j\le n$. There are four summands in the right hand side of the recursion relation. Let us denote 
them by $S_1$, $S_2$, $S_3$, $S_4$ in the same order as they are listed in proposition~\ref{prop:recursion}. We prove
that $S_i\stackrel{a_E}{\equiv}P_i(a,a_{\{1,\dots,n\}\setminus E},\sum_{j\in E}a_j)$ for some $P_i$ separately for each $i$.
It is easy to see that $S_1\stackrel{a_E}{\equiv}S_2\stackrel{a_E}{\equiv}0$. 

Let us discuss $S_3$. It can be represented as a sum
$$
S_3=\sum_{J_1\subset \{1,\dots,n\}\setminus E}\sum_{k\leq |E|} S_{J_1,k},
$$
 where
\begin{align*}
& S_{J_1,k}=
\sum_{\substack{J_2\subset E\\|J_2|=k}}\left(a+\sum_{j\in J_1\sqcup J_2}a_j\right)\cdot\left(2g+|\{1,\dots,n\}\setminus(J_1\sqcup J_2)|\right)\cdot \\
& \left\langle
\genfrac[]{0pt}{}{a+\sum_{j\in J_1\sqcup J_2}a_j}{0}\prod_{i\in \{1,\dots,n\}\setminus(J_1\sqcup J_2)} \genfrac[]{0pt}{}{ a_i}{d_i}
\right\rangle^{DR}_g
\left\langle
\genfrac[]{0pt}{}{ a}{d} \prod_{j\in J_1\sqcup J_2} \genfrac[]{0pt}{}{ a_j}{d_j}
\right\rangle^{DR}_0 
\end{align*}

Let us prove that $S_{J_1,k}\stackrel{a_E}{\equiv}P_{J_1,k}(a,a_{\{1,\dots,n\}\setminus E},\sum_{j\in E}a_j)$ for some $P_{J_1,k}$. 
Note that $\left\langle\genfrac[]{0pt}{}{ a}{d} \prod_{j\in J_1\sqcup J_2} \genfrac[]{0pt}{}{ a_j}{d_j}\right\rangle^{DR}_0$ and $\left(2g+|\{1,\dots,n\}\setminus(J_1\sqcup J_2)|\right)$ are just some constants that depend only on the subset $J_1$ and the number $k=|J_2|$. 
The induction assumption implies that
\begin{align*}
& \left\langle\genfrac[]{0pt}{}{b}{0}\prod_{i=1}^{|E|-k}\genfrac[]{0pt}{}{x_i}{0}\prod_{i\in \{1,\ldots,n\}\backslash(E\sqcup J_1)}\genfrac[]{0pt}{}{ a_i}{d_i}\right\rangle^{DR}_g\\
& \stackrel{b,x}{\equiv}Q_{J_1,k}\left(a_{\{1,\ldots,n\}\backslash(E\sqcup J_1)},b+\sum_{i=1}^{|E|-k}x_i\right)
\end{align*} 
for some polynomial $Q_{J_1,k}$. Hence
\begin{align*}
&\left(a+\sum_{j\in J_1\sqcup J_2}a_j\right)\left\langle\genfrac[]{0pt}{}{a+\sum_{j\in J_1\sqcup J_2}a_j}{0}\prod_{i\in\{1,\ldots,n\}\backslash(J_1\sqcup J_2)}\genfrac[]{0pt}{}{a_i}{d_i}\right\rangle^{DR}_g\\
&\stackrel{a_E}{\equiv}\left(a+\sum_{j\in J_1\sqcup J_2}a_j\right)Q_{J_1,k}\left(a_{\{1,\ldots,n\}\backslash(E\sqcup J_1)},a+\sum_{i\in J_1\sqcup E}a_i\right).
\end{align*}
Notice that
$$
\sum_{\substack{J_2\subset E\\|J_2|=k}}\left(a+\sum_{j\in J_1\sqcup J_2}a_j\right)=
\binom{|E|}{k}\left(a+\sum_{j\in J_1}a_j\right)+\binom{|E|-1}{k-1}\sum_{j\in E}a_j.
$$
Therefore, 
\begin{align*}
Q_{J_1,k}(a_{\{1,\ldots,n\}\backslash(E\sqcup J_1)},a+\sum_{i\in J_1\sqcup E}a_i)\cdot\sum_{\substack{J_2\subset E\\|J_2|=k}}\left(a+\sum_{j\in J_1\sqcup J_2}a_j\right)
\end{align*}
can be represented as a polynomial that dependes only on $a$, $a_i$, $i\in \{1,\dots,n\}\setminus E$, and $\sum_{j\in E}a_j$.
The same argument can be applied to $S_4$. This concludes the proof of the lemma.
\end{proof}

\begin{proof}[Proof of lemma~\ref{lemma:coef3}]
The lemma follows immediately from proposition~\ref{prop:divisibility}. Indeed, the statement of the lemma is equivalent to the fact 
that the polynomial~\eqref{eq:divisibility} doesn't have terms linear in the variable $b$.  
\end{proof}


\end{document}